\documentclass[12pt,a4paper]{amsart}
\usepackage{mathrsfs}
\usepackage{amsfonts}
\usepackage[active]{srcltx}
\usepackage[notref,notcite]{showkeys}
\usepackage{enumerate}

\usepackage{amsmath,amssymb,xspace,amsthm}
\newtheorem{theorem}{Theorem}
\newtheorem{proposition}[theorem]{Proposition}
\newtheorem{corollary}[theorem]{Corollary}
\numberwithin{equation}{section}

\def\span{\operatorname{span}}

\newcommand{\C}{\ensuremath{\mathbb C}\xspace}

\renewcommand{\i}{\ensuremath{\mathfrak i}}

\newcommand{\Z}{\ensuremath{\mathbb{Z}}\xspace}
\newcommand{\N}{\ensuremath{\mathbb{N}}\xspace}

\newcommand{\V}{\ensuremath{\mathfrak{V}}\xspace}

\renewcommand{\phi}{\varphi}

\renewcommand{\leq}{\leqslant}
\renewcommand{\geq}{\geqslant}

\begin{document}
\title{Irreducible Virasoro modules from tensor products}
\author{Haijun Tan and  Kaiming Zhao}
\maketitle

\begin{abstract}
   In this paper, we obtain a class of irreducible Virasoro modules
    by taking tensor products of the irreducible Virasoro modules $\Omega(\lambda,b)$ defined in \cite{LZ},
    with irreducible highest weight modules
   $V(\theta ,h)$ or with irreducible Virasoro modules Ind$_{\theta}(N)$ defined in \cite{MZ2}.
    We
    determine the necessary and sufficient conditions for two such irreducible tensor products to be isomorphic.
    Then we prove that the tensor product of   $\Omega(\lambda,b)$ with
    a classical Whittaker module is isomorphic to
 the module $\mathrm{Ind}_{\theta,
\lambda}(\mathbb{C_\mathbf{m}})$ defined in \cite{MW}. As a
by-product we obtain  the necessary and sufficient conditions for
the module $\mathrm{Ind}_{\theta,
     \lambda}(\mathbb{C_\mathbf{m}})$ to be irreducible. We also generalize the module $\mathrm{Ind}_{\theta,
\lambda}(\mathbb{C_\mathbf{m}})$ to
$\mathrm{Ind}_{\theta,\lambda}(\mathcal{B}^{(n)}_{\mathbf{s}})$ for
any non-negative integer $ n$ and use the above results to
completely determine when the modules
$\mathrm{Ind}_{\theta,\lambda}(\mathcal{B}^{(n)}_{\mathbf{s}})$ are
irreducible. The submodules of
$\mathrm{Ind}_{\theta,\lambda}(\mathcal{B}^{(n)}_{\mathbf{s}})$ are
studied and an open problem in \cite{GLZ} is solved.  Feigin-Fuchs'
Theorem on singular vectors of Verma modules over the Virasoro
algebra is crucial to our proofs in this paper.
\end{abstract}


\vskip 5pt \noindent {\em 2000  Math. Subj. Class.:} 17B10, 17B20,
17B65, 17B66, 17B68

\vskip 10pt

\section{Introduction}

\vskip 5pt We denote by $\mathbb{Z}$, $\mathbb{Z}_+$, $\mathbb{N},$
$\mathbb{R}$ and $\mathbb{C}$
   the sets of  all integers, non-negative integers, positive integers, real numbers and complex numbers,
   respectively.

\vskip 5pt The \textbf{Virasoro algebra} $\V$ := Vir[$\mathbb{Z}$]
(over
   $\mathbb{C}$) is the Lie algebra with the basis $\{c,d_i|i\in \mathbb{Z}\}$
   and the Lie brackets defined by
   $$[c,d_i]=0,\,\,\,\,[d_i,d_j]=(j-i)d_{i+j}+\delta_{i,-j}\frac{i^3-i}{12}c,
    \forall i,j \in \mathbb{Z}.$$
   The algebra  $\V$ is one of the most important Lie algebras both in mathematics
   and in mathematical physics, see for example \cite{KR, IK} and references therein.
   The Virasoro algebra theory has been widely used in many physics areas and
   other mathematical brances, for example, quantum physics \cite{GO}, conformal field
   theory \cite{FMS}, higher-dimensional WZW models \cite{IKUX, IKU}, Kac-Moody algebras
   \cite{K2, MoP}, vertex algebras \cite{LL}, and so on.

\vskip 5pt   The theory of weight Virasoro modules with
finite-dimensional weight spaces is
   fairly well developed. In particular, a classification of weight Virasoro modules
   with finite-dimensional weight spaces was given by Mathieu \cite{M}, and a classification
   of weight Virasoro modules with at least one finite dimensional nonzero weight space was given in
   \cite{MZ1}. There are some known irreducible weight Virasoro modules with infinite dimensional
   weight spaces, see \cite{Zh,CM,CGZ,LLZ}. We remark that the tensor
   product of intermediate
   series modules over the Virasoro algebra is never  irreducible
   \cite{Zk}.
   For non-weight irreducible Virasoro modules, there are
   Whittaker modules, see \cite{OW} and \cite{BM}, and other non-Whittaker modules, see \cite{MW,MZ2,LZ,LLZ}.

\vskip 5pt The purpose of the present paper is to construct new
irreducible (non-weight) Virasoro modules by taking tensor product
of some known irreducible Virasoro modules recently defined in
\cite{LZ} and \cite{MZ2}.    Let us first recall some notions and
   results which will be used later.

\vskip 5pt
     For any pair $(\lambda, b)\in \mathbb{C}^\ast \times
     \mathbb{C}$, the Virasoro module $\Omega(\lambda,
     b)$ is defined on the polynomial (associative) algebra
      $\mathbb{C}[\partial] $  in one indeterminant
   $\partial$ over $\mathbb{C}$ with the action of $\V $   given by
     $$d_n\partial ^j = \lambda^n(\partial + n(b-1))(\partial-n)^j,\,\,\,\, c\partial ^j=0,
      \forall j\in \mathbb{Z}_+, n\in\mathbb{Z}.$$
     It was proved in \cite{LZ} that $\Omega(\lambda,  b)$ is
     irreducible if and only if $b\ne 1$; if $b=1$ then $\Omega(\lambda,
     1)$ has a codimension one irreducible submodule isomorphic to $\Omega(\lambda,
     0)$.

\vskip 5pt Let $U:=U(\V)$ be the universal enveloping algebra of the
Virasoro algebra $\V$. For any $\theta, h\in \C$, let $I(\theta,h)$
be the left ideal of $U$ generated by the set $$
\bigl\{d_{i}\bigm|i>0\bigr\}\bigcup\bigl\{d_0-h\cdot 1, c-\theta
\cdot 1\bigr\}. $$ The Verma module with highest weight $(\theta,
h)$ for $\V $ is defined as the quotient  $\bar V(\theta,
h):=U/I(\theta,h)$. It is a highest weight module of $\V $ and has a
basis consisting of all vectors of the form $$
d^{k_{-1}}_{-1}d^{k_{-2}}_{-2}\cdots d^{k_{-n}}_{-n}v_{h};\quad
k_{-1},k_{-2},\cdots, k_{-n} \in{\Z_+}, n\in \N,$$ where
$v_h=1+I(\theta ,h)$. Each nonzero scalar multiple of $v_h$ is
called a highest weight vector of the Verma module. Then we have the
{\it irreducible highest weight module} $ V(\theta,h)=\bar V(\theta,
h)/J,$ where $J$ is the maximal proper submodule of $\bar V(\theta,
h)$. For the structure of $ V(\theta, h)$, refer to \cite{FF}
 (refined versions are in \cite{A, D}).

\vskip 5pt Denote by $\V _+$ the Lie subalgebra of $\V$ spanned by
all $d_i$ with
     $i\geq 0$. For $n\in \mathbb{Z}_+$, denote by $\V _+^{(n)}$ the Lie subalgebra
     of $\V $ generated by all $d_i$ for $ i>n$.  For any $\V _+$ module $N$ and
     $\theta \in \mathbb{C}$, consider the induced module
     $\mathrm{Ind}(N):=U(\V)\otimes_{U(\V _+)}N$,
     and denote by $\mathrm{Ind}_{\theta}(N)$ the module
     $\mathrm{Ind}(N)/(c-\theta)\mathrm{Ind}(N)$.
     From [MZ2] we know that for an irreducible $\V _+$ module $N$, if there
     exists $k\in \mathbb{N} $ such that $d_k$ acts injectively on $N$ and
     $d_iN=0$ for all $i>k$, then
     $\mathrm{Ind}_{\theta}(N)$ is an irreducible $\V$ module
     for any $\theta \in \mathbb{C}$.

\vskip 5pt The present paper is organized as follows. In Section 2,
we obtain a
    class of irreducible non-weight modules by taking the tensor product of
    $\Omega(\lambda, b)$ with the highest weight module $V(\theta ,h)$ or with
    the modules $\mathrm{Ind}_{\theta}(N)$ (see Theorem 1). In Section 3, we determine the necessary and
    sufficient conditions  for two irreducible modules $\Omega(\lambda, b)\otimes V$ and
    $\Omega(\lambda', b')\otimes V'$ to be isomorphic (Theorem 2).
    In Section 4, we compare the tensor product
    modules $\Omega(\lambda,b)\otimes V$ with all other known non-weight irreducible modules in \cite{LZ,LLZ,MZ2,MW}.
    In particular, we prove that the
    tensor product of $\Omega(\lambda,b)$ with the classical Whittaker module (see \cite{OW} and \cite{LGZ})
    is isomorphic to
    the module $\mathrm{Ind}_{\theta,\lambda}(\mathbb{C}_{\mathbf{m}})$ defined in \cite{MW}.
    As a by-product, we obtain the necessary and sufficient conditions for the modules
    $\mathrm{Ind}_{\theta,\lambda}(\mathbb{C}_{\mathbf{m}})$ to be irreducible (theorem 5). From these we conclude
    that the modules $\Omega(\lambda,b)\otimes V$ are new when $V$ are not the classical irreducible
     Whittaker modules (Proposition 6). In section 5, we generalize the modules
     $\mathrm{Ind}_{\theta,\lambda}(\mathbb{C}_{\mathbf{m}})$ which were defined and studied in \cite{MW} to the modules
     $\mathrm{Ind}_{\theta,\lambda}(\mathcal{B}^{(n)}_{\mathbf{s}})$
     for any
 $n\in\Z_+$. More precisely, for $n\in\Z_+, \lambda\in
\C^*$, first we define   the subalgebra of $\V$ as follows
$$\mathfrak{b}_{\lambda,n}=\mathrm{span}_{\C}\{d_k-\lambda^{k-n+1}d_{n-1}: k\geq n \}.$$  For any $\theta\in \C$ and
 ${\bf{s}}=(s_n, s_{n+1},\cdots, s_{2n})\in \C^{n+1}$, we define the
 $1$-dimensional $\mathfrak{b}_{\lambda,n}$-module $\mathcal{B}^{(n)}_{\mathbf{s}}$
    on $\C$
     over
     $\mathfrak{b}_{\lambda,n}$  by
     $$(d_k-\lambda^{k-n+1}d_{n-1})\cdot 1=s_k, \ n\leq k\leq 2n.$$
Then we have our
       Virasoro module
    $$ \mathrm{Ind}_{\theta, \lambda}(\mathcal {B}^{(n)} _{{\bf {s}}}):=
    \left(\mathrm{Ind}_{\mathfrak{b}_{\lambda,n}}^{\V}\mathcal{B}^{(n)}_{\mathbf{s}}\right )
    /(c-\theta)\left(\mathrm{Ind}_{\mathfrak{b}_{\lambda,n}}^{\V}\mathcal{B}^{(n)}_{\mathbf{s}}\right ).$$
We  use the above established results to  obtain the necessary and
     sufficient conditions for the modules $\mathrm{Ind}_{\theta,\lambda}(\mathcal{B}^{(n)}_{\mathbf{s}})$
     to be irreducible in Theorems 7, 8 and
     9  for different cases of  $n$.  We remark that the three cases are
totally different. We also study the submodules of
$\mathrm{Ind}_{\theta,\lambda}(\mathcal{B}^{(n)}_ {\mathbf{s}})$ in
Theorem 10. As a by-product, Corollary 11 solves the
     open problem in \cite{GLZ}, i.e., $\mathrm{Ind}_{\theta, \lambda}(\mathcal {B}^{(0)} _{{\bf{s}}})$ has
a unique maximal submodule. Our main technique used in this paper is
Feigin-Fuchs' Thereom in \cite{FF} on singular vectors of Verma
modules over the Virasoro algebra.

In the subsequent paper \cite{TZ} we generalize all the above
results. In particular, we determine necessary and sufficient
conditions for the tensor product of finitely many modules of the
form $\Omega(\lambda,b)$ to be simple.

\section{ Constructing Non-weight Modules}
\vskip 5pt In this section we will obtain a class of irreducible non-weight modules over $\V$ by
  taking the tensor products of $\Omega(\lambda,b)$ with two classes of other modules, which is the following

\begin{theorem} Let $\lambda\in \C^*$ and $b\in  \C\setminus\{1\}$.
Let $V$ be an irreducible module over $\V$ such that each $d_k$ is
locally finite on $V$ for all $k\ge R$ where $R$ is  a fixed
positive integer. Then $\Omega(\lambda,b)\otimes V$ is an
irreducible Virasoro module.
\end{theorem}

\begin{proof} From Theorem 2 in \cite{MZ2} we know that $V$ has to be
$V(\theta, h)$ for some $\theta, h\in\C$ or
$\mathrm{Ind}_{\theta}(N)$ defined in \cite{MZ2}. Let
$W=\Omega(\lambda,b)\otimes V$.
    It is clear that,      for any $v\in V$, there is  a  positive
    integer     $K(v)$ such that $d_l(v)=0$ for all $l\geq K(v)$.
\vskip 5pt Suppose $M$ is a nonzero submodule of $W$. It suffices to
    show that  $M=W$. Take a nonzero  $w=\sum_{j=0}^{s}\partial^j\otimes
    v_j\in W$ such that
      $v_j\in V$, $v_s\neq 0$ and $s$ is minimal.

\vskip 5pt {\bf Claim 1.} $s=0$.

\vskip 5pt Let $K=\max\{K(v_j):j=0,1,\cdots, s\}$.
Using $d_l(v_j)=0$ for all $l\geq K$
      and $j=0,1,\cdots,s$, we deduce that
$$\lambda^{-l}d_lw=\sum_{j=0}^{s}(\partial + l(b-1))(\partial-l)^j\otimes
    v_j\in M, \,\,\forall l\ge K.$$
    Write the right hand side as
    \begin{equation}\sum_{j=0}^{s+1}l^jw_i\in M, \,\,\forall l\ge K,\end{equation}
    where $w_i\in W$ are independent of $l$. In particular,  $w_{s+1}= (b-1)(-1)^{s-1}\otimes
    v_s$. Taking $l=K, K+1, ..., K+s$, we see that the coefficient
    matrix of $w_i$ is a Vandermonde matrix.
So each $w_i\in M$. In particular, $w_{s+1}= (b-1)(-1)^{s-1}\otimes
v_s\in M$. Consequently $s=0$.

\vskip 5pt {\bf Claim 2.} $M=W$.

 \vskip 5pt  From Claim 1 we know that  $1\otimes v\in M$
 for some nonzero $v\in V$.
By induction on $t$ and using
   $$\aligned d_l(\partial^{t}\otimes v)&=
      (\lambda^l(\partial + l(b-1))(\partial-l)^t)\otimes v\\
      &=\lambda^l(\partial-l)^{t+1}\otimes  v+
      lb\lambda^l(\partial-l)^{t}\otimes  v,\endaligned$$
   where $l\geq K(v), t\in \mathbb{Z}_+$, we  deduce that $\partial^t \otimes v \in M$ for all $t\in
   \mathbb{Z}_+$, i.e.,
   $\Omega(\lambda,b)\otimes v\subset M$. Let $X$ be a maximal
   subspace of $V$ such that $\Omega(\lambda,b)\otimes X\subset M$. We
   know that $X\ne0$. Clearly, $X$ is a nonzero submodule of $V$.
   Since $V$ is irreducible, we obtain that $X=V$. Therefore, $M=W$ and $W$ is irreducible.
\end{proof}

{\bf Example 1.} Let $\lambda_1, \lambda_2,\theta\in \mathbb{C}$ and
let $J$ be the left ideal of $U(\V _+)$
    generated by $d_1-\lambda_1, d_2-\lambda_2, d_3,d_4, ....$. We define $N:=U(\V _+)/J$.
    Then $V=\mathrm{Ind}_{\theta}(N)$ is the classical Whittaker module (See \cite{OW} or \cite{MZ2}).
    From \cite{LGZ}
    we know that if $\lambda_1\neq 0$ or $\lambda_2\neq 0$, then
    $V$ is both an irreducible $\V$ module and a locally nilpotent $\V _+^{(2)}$ module.
    By Theorem 1 we know that $\Omega(\lambda,b)\otimes V$ is irreducible for any $\lambda\in
    \C^*$ and $b\in  \C\setminus\{1\} $. These modules will be studied in detail  in section 4.

\section {Isomorphisms}
\vskip 5pt In this section we will determine the necessary and
sufficient conditions for two irreducible tensor products   defined
in Theorem 1 to be isomorphic, which is the following

\begin{theorem} Let $\lambda, \lambda'  \in \C^*$, $b, b'\in
\C\setminus\{1\}$,
  and let $V$ and $V'$ be two irreducible modules over $\V$ such that each $d_k$ is
locally finite on both $V$ and $V'$ for all $k\ge R$ where $R$ is  a
fixed positive integer.
  Then $W=\Omega(\lambda, b)\otimes V$ and $W'=\Omega(\lambda', b')\otimes V'$ are isomorphic as $\V$
  modules if and only if
  $(\lambda, b)=(\lambda',b')$ and $V\cong V'$ as $\V$ modules.
\end{theorem}

\begin{proof}
  The sufficiency of the Theorem is obvious. We need only to prove the necessity.
 Let $\phi$ be an isomorphism from $W$ to $W'$.

\vskip 5pt Take a nonzero element $1\otimes  v \in W$. Suppose
 $$\phi(1\otimes  v)=\sum_{j=0}^{n}\partial^j\otimes  w_j,$$
  where ${w_j}\in {V'}$ with $w_n\neq 0$. There is a positive integer $K=K(v) $ such that
  $d_l(v)=d_l({w_j})=0$ for all integers $l\geq K$ and $0\leq j\leq
  n$. For any  $l,l'\geq K,$ we have
$$(\lambda^{-l}d_l-\lambda^{-l'}d_{l'})(1\otimes
v)=(l-l')(b-1)(1\otimes  v).$$ Then $$\aligned
 &(l-l')(b-1)\sum_{j=0}^{n}\partial^j\otimes  w_j=(\lambda^{-l}d_l-\lambda^{-l'}d_{l'})\sum_{j=0}^{n}
 (\partial^j\otimes  {w_j})\\
=&\sum_{j=0}^{n}((\frac{\lambda'}{\lambda})^l(\partial +
l(b'-1))(\partial-l)^j
  -(\frac{\lambda'}{\lambda})^{l'}(\partial + l'(b'-1))(\partial-l')^j)\otimes
  {w_j}.
  \endaligned$$
   We deduce that
  $$((\frac{\lambda'}{\lambda})^l-(\frac{\lambda'}{\lambda})^{l'})(\partial^{n+1}\otimes w_n)=0, \,\,\,
  \forall \,\,\,l,l'\geq K.$$
   So
    $\lambda'=\lambda$. The previous equation becomes
$$
 \aligned  (l-l')&(b-1)\sum_{j=0}^{n}\partial^j\otimes  w_j \\
&=\sum_{j=0}^{n} (b'-1)(l(\partial-l)^j
  -l'(\partial-l')^j)\otimes
  {w_j}\\ &\hskip 1cm +\sum^n_{j=0}\partial ((\partial-l)^j-(\partial - l')^j)\otimes w_j,
\endaligned
$$
  where $l,l'\geq K$. If $n>0$, the coefficient of $l^{n+1}$ is $(b'-1)1\otimes w_n$
which is nonzero, yielding a contradiction. So  $n=0$, hence $
b'=b.$
  Thus there is a one to one and onto linear map $\tau: V\to V'$ such
that \begin{equation}\phi(1\otimes  v)= 1\otimes
\tau(v),\,\,\forall v\in V.\end{equation}

\vskip 5pt \vskip 5pt Since
$$\phi(d_l(1\otimes v))=d_l(\phi(1\otimes v)), \,\,\forall\,\,l\geq K,
$$
that is,
$$\lambda^l\phi(\partial\otimes v)+\lambda^ll(b-1)(1\otimes \tau (v))$$
$$ =\lambda^l(\partial \otimes \tau (v))+\lambda^ll(b-1)(1\otimes \tau (v)),$$
we see that $\phi (\partial \otimes v)= \partial \otimes \tau (v).$
Hence, $\phi(d_j(1)\otimes v) =d_j(1)\otimes \tau (v), j\in
\mathbb{Z}$. From $\phi(d_j(1\otimes v))=d_j(\phi (1\otimes v)),
j\in \mathbb{Z}$ we can deduce that $\phi (1\otimes d_j(v))=
1\otimes d_j(\tau (v))$. So $$\tau(d_j(v))=d_j(\tau (v)),\,\,
\forall\,\,j\in \mathbb{Z}, v\in V.$$ Clearly, $\phi (c(1\otimes
v))=c(\phi(1\otimes v))$ implies that $\tau (cv)=c\tau(v)$. Thus
$\tau : V\rightarrow V'$  is a $\V$ module isomorphism and $V\cong V'$. This
completes the proof.
\end{proof}

\section{ New Irreducible Modules $\Omega(\lambda,b)\otimes V$ }

\vskip 5pt In this section we will  compare the irreducible tensor
products in Theorem 1 with all other known non-weight irreducible
Virasoro modules in  \cite{LZ},  \cite{LLZ},  \cite{MZ2} and
\cite{MW}.

 For any $s\in\Z_+, l,m\in\Z$, as in \cite{LLZ}, we denote
   $$\omega_{l,m}^{(s)}=\sum^s_{i=0}\begin{pmatrix} s \\i \\ \end{pmatrix}
  (-1)^{s-i}d_{l-m-i}d_{m+i}\in U(\V).$$

\begin{proposition} Let $\lambda  \in \C^*$, $b\in
\C\setminus\{1\}$,
  and let $V$  be an infinite dimensional irreducible $\V$ module such that $V$ is
  an irreducible module over $\V$ such that each $d_k$ is
locally finite on $V$ for all $k\ge R$ for a fixed $R\in\N$.

 \begin{enumerate}[$(i).$]
\item  For any positive integer $n$, the action of $\V_+^{(n)}$ on
$\Omega(\lambda,b)\otimes V$ is not locally finite.
\item  For any integer $s>4$, there exists $v\in V, m, l\in\Z$ such
that in $\Omega(\lambda,b)\otimes V$ we have
$$\omega_{l,-m}^{(s)}(1\otimes v)\ne0.$$\end{enumerate}
\end{proposition}

\begin{proof} As we mentioned before, $V$ has to be
$V(\theta, h)$ for some $\theta, h\in\C$ or
$\mathrm{Ind}_{\theta}(N)$ defined in \cite{MZ2}. Clearly, the
module $\Omega(\lambda,b)\otimes V$ is not a weight module. Part
$(i)$ follows from considering   $d_{n+1}^k(1\otimes v)$ for any
nonzero $v\in V$ and any $k\in\N$.

\vskip 5pt $(ii)$.  Take $v$ to be the highest weight vector of $V$
if $V$ is a highest weight module, otherwise $v$ can be any nonzero
vector of $V$. From  \cite{FF} and \cite{MZ2} we know that the set
$S=\{v, d_{-2}v, d_{-3}v,,...,d_{-s-2}v \}$ is linearly independent
and $d_lS=0$ for $l>K$ where $K\in\N$ is an integer depending on $v$
and $s$. Then for any $l>K$ and $ m=s+2$, noting that
$\omega_{l,-m}^{(s)}(1)=0$ in $\Omega(\lambda, b)$ we deduce that
$$\omega_{l,-m}^{(s)}(1\otimes v)=\sum^s_{i=0}{s \choose i}
  (-1)^{s-i}d_{l+m-i}d_{-m+i}(1\otimes v)$$
$$=\sum^s_{i=0}{s \choose i}
  (-1)^{s-i}d_{l+m-i}(1)\otimes d_{-m+i}(v)$$
$$=\sum^s_{i=0}{s \choose i}
  (-1)^{s-i}\lambda^{l+m-i}(\partial + (l+m-i)(b-1))\otimes d_{-m+i}(v),$$
which is nonzero. \end{proof}

\begin{corollary} Let $\lambda  \in \C^*$, $b\in
\C\setminus\{1\}$,
  and let $V$  be an infinite dimensional irreducible $\V$ module such that $V$ is
  an irreducible module over $\V$ such that each $d_k$ is
locally finite on $V$ for all $k\ge R$ for a fixed $R\in\N$. 
  Then $\Omega(\lambda, b)\otimes V$ is not isomorphic to any irreducible module
defined in \cite{LZ},  \cite{LLZ}, or in \cite{MZ2}.
\end{corollary}

\begin{proof} Since there is  an $n\in\N$ such that the action of $\V_+^{(n)}$ on the modules
$\mathrm{Ind}(N)$ defined in \cite{MZ2} is  locally finite, from
Proposition 3(i) we see  that $\Omega(\lambda, b)\otimes V$ is not
isomorphic to any  module described in \cite{MZ2}.

\vskip  5pt  Now let us consider an irreducible Virasoro module
$A_{b'}$ defined in \cite{LZ} where  $b'\in \mathbb{C}$ and $A$ is
an irreducible module over the associative algebra $\C[t, t^{-1},
t\frac d{dt}]$. We may assume that $A$ is $\C[t\frac
d{dt}]$-torsion-free, otherwise $A_{b'}$ will be a weight Virasoro
module. The action on $A_{b'}$ is
   $$cw=0,\,\,\,\,d_nw=(t^n\partial+nbt^n)w, \forall n\in \mathbb{Z}, w\in A,$$
   where $\partial=t\frac d{dt}$ and the left hand side is associative algebra action.
  From the proof of Theorem 9 in \cite{LLZ}, we know that $$\omega_{l,m}^{(s)}(A_{b'})=0,\,\,\,\forall\,\,
  l,m\in\Z, s\ge 3.$$
From Proposition 3(ii) we see that
 $\Omega(\lambda, b)\otimes V$ is not isomorphic to the   modules $A_{b'}$.

\vskip 5pt Let us consider the modules $\mathcal{N}(M, \beta)$
described in [LLZ], where $M$ is an irreducible module over the Lie
algebra $\mathfrak{a}_r:={\V _+}/{\V ^{(r)}_{+}}, r\in \N$ such that
the action of $\bar{d_r}:=d_r+\V _+^{(r)}$ on $M$ is injective,
$\beta \in \mathbb{C}[t^{\pm}]\setminus\mathbb{C}$. We know that
$\mathcal{N}(M, \beta)=M \otimes \mathbb{C}[t^{\pm}]$, and the
action of $\V$ on $\mathcal{N}(M, \beta)$ is defined by
$$d_m\circ (v\otimes t^n) =(n+\sum _{i=0}^{r}(\frac{m^{i+1}}{(i+1)!}\bar{d_i})v)
\otimes t^{n+m}+v\otimes (\beta t^{m+n}), $$
$$c(v\otimes t^n)=0, m,n\in \mathbb{Z}.$$
Note that   if $r=0$ the modules $\mathcal{N}(M, \beta)$ with $\beta
\in \mathbb{C}[t^{\pm}]\setminus\mathbb{C}$ are some modules of the
form $A_{b'}$ (see the beginning of Section 6 in \cite{LLZ} and
Section 4.1 in \cite{LZ}).  From the computation in (6.7) of \cite{LLZ}
we see that
$$\omega_{l,m}^{(s)} (\mathcal{N}(M,\beta))=0, \forall\,\, l,m\in\Z,
s>2r+2.$$ From Proposition 3(ii) we see that
 $\Omega(\lambda, b)\otimes V$ is not isomorphic to the  modules $\mathcal{N}(M, \beta)$.
\end{proof}

\vskip 5pt Now we compare our modules $\Omega(\lambda, b)\otimes V$
in Theorem 1 with the modules
   $\mathrm{Ind}_{\theta, \lambda}(\mathbb{C_\mathbf{m}})$ defined in \cite{MW}. Let us first recall
   the definition for $\mathrm{Ind}_{\theta, \lambda}(\mathbb{C_\mathbf{m}})$ from  \cite{MW}.
\vskip 5pt Let $\lambda\in \mathbb{C}^*$, denote by
$\mathfrak{b}_\lambda$ the subalgebra of $\V$ generated by
$d_k-\lambda^{k-1}d_1, k\geq 2$. For a fixed 3-tuple
$\mathbf{m}=(m_2,m_3,m_4)\in \mathbb{C}^3$, we define the action of
$\mathfrak{b}_\lambda$ on $\mathbb{C}$ by
\begin{equation}
\begin{split}
&(d_k-\lambda^{k-1}d_1)\cdot 1=m_k, k=2,3,4;\\
& (d_k-\lambda^{k-1}d_1)\cdot
1=(k-3)m_4\lambda^{k-4}-(k-4)m_3\lambda^{k-3}, k>4.
\end{split}
\end{equation}
This gives a $\mathfrak{b}_\lambda$ module construction on
$\mathbb{C}$. We denote it by $\mathbb{C}_{\mathbf{m}}$. Note that
the second equation in (4.1) follows from the first. For a fixed
$\theta\in \mathbb{C}$, the module $\mathrm{Ind}_{\theta,
\lambda}(\mathbb{C_\mathbf{m}})$ is defined as follows
\begin{equation}
\mathrm{Ind}_{\theta,
\lambda}(\mathbb{C_\mathbf{m}}):=U(\V)\otimes_{U(\mathfrak{b}_\lambda)}\mathbb{C}_{\mathbf{m}}/
  (c-\theta)U(\V)\otimes_{U(\mathfrak{b}_\lambda)}\mathbb{C}_{\mathbf{m}}.
\end{equation}
From the Theorem 1 in \cite{MW} we know that the $\V$ module
$\mathrm{Ind}_{\theta, \lambda}(\mathbb{C_\mathbf{m}})$ is
irreducible if $(m_2,m_3,m_4)\in \mathbb{C}^3$ and $\lambda\in
\mathbb{C}\backslash\{0\}$ satisfy the following conditions
\begin{equation}
\lambda m_3\neq m_4, 2\lambda m_2\neq m_3, 3\lambda m_3\neq 2m_4,
\lambda^2m_2+m_4\neq 2\lambda m_3.
\end{equation}

\vskip 5pt Now we are ready to prove the following

\begin{theorem}
Let $\lambda \in \C^*$,   $\lambda_1, \lambda_2, \theta, b\in\C$,
  and let $V$ be the classical irreducible Whittaker module
  $U(\V)/I$ where $I$ is the left ideal of $U(\V)$ generated by $$c-\theta, d_1-\lambda_1, d_2-\lambda_2, d_3,d_4,
  ....$$
  \begin{enumerate}[$(i).$]
\item The module
 $\Omega(\lambda, b)\otimes V $ is isomorphic to $ \mathrm{Ind}_{\theta, \lambda}(\mathbb{C_\mathbf{m}}),$
  where
\begin{equation}
\begin{split}
&m_2=\lambda_2-\lambda\lambda_1+\lambda^2(b-1),\\
& m_3=\lambda^{2}(-\lambda_1+2\lambda(b-1)),\\
&m_4=\lambda^{3}(-\lambda_1+3\lambda(b-1)).
\end{split}
\end{equation}
\item The module $\mathrm{Ind}_{\theta,
\lambda}(\mathbb{C_\mathbf{m}})$ is irreducible if and only if
$\lambda m_3\neq m_4$, and $3\lambda m_3\neq 2m_4$ or
$\lambda^2m_2+m_4\neq 2\lambda m_3.$
\end{enumerate}
\end{theorem}

\begin{proof} Let us denote $W=\Omega(\lambda, b)\otimes V$ and $v=1+I$ in $V$.
By simple computations and using Theorem 1 we have that
\begin{enumerate}[$(a).$]
\item $W$ is cyclic with generator $1\otimes v$;
\item $W$ is irreducible if and only if  $b\ne1$, and $\lambda_1\ne0$ or
$\lambda_2\ne0$.
\end{enumerate}

\

In $W$, for $k=2,3,4$ we compute that
$$(d_k-\lambda^{k-1}d_1)(1\otimes v)=(d_k-\lambda^{k-1}d_1)(1)\otimes v+ 1\otimes (d_k-\lambda^{k-1}d_1)(v)$$
$$=\lambda^k((\partial+k(b-1))-(\partial+(b-1)))\otimes v+1\otimes(\delta_{k,2}\lambda_2-\lambda^{k-1}
\lambda_1)(v)$$
$$=(\lambda^k(k-1)(b-1)-\lambda^{k-1}\lambda_1+\delta_{k,2}\lambda_2)(1\otimes v) $$$$=m_k(1\otimes v),$$
where $m_k, k=2,3,4$ are given by (4.4). It follows that
$$(d_k-\lambda^{k-1}d_1)(1\otimes v)=((k-3)m_4\lambda^{k-4}-(k-4)m_3\lambda^{k-3})(1\otimes v),
$$ for all $k>4.$ Because of the universal property of the module $\mathrm{Ind}_{\theta, \lambda}(\mathbb{C_\mathbf{m}})$
we have the following surjective (onto) homomorphism of
modules
$$\tau: \mathrm{Ind}_{\theta, \lambda}(\mathbb{C_\mathbf{m}}) \to W,$$
uniquely determined by $\tau(1)=1\otimes v$.

\vskip 5pt It is clear that $\mathrm{Ind}_{\theta,
\lambda}(\mathbb{C_\mathbf{m}})$ has a basis
$$\{d_{-n}^{k_{-n}}\cdots d_{-1}^{k_{-1}}d_0^{k_0}d_1^{k_1}\cdot 1: k_1,k_0, k_{-1},
 \cdots, k_{-n}\in \mathbb{Z}_+, n\in \mathbb{Z}_+\}.$$ Then $W$ is the linear span of the set
\begin{equation}
\{d_{-n}^{k_{-n}}\cdots d_{-1}^{k_{-1}}d_0^{k_0}d_1^{k_1}\cdot (1\otimes v):
k_1,k_0, k_{-1},
 \cdots, k_{-n}\in \mathbb{Z}_+, n\in \mathbb{Z}_+\}
 \end{equation} because $\tau $ is surjective. We know that $W$ has
 a basis
$$B=\{\partial^{k_1}\otimes \left(d_{-n}^{k_{-n}}\cdots d_{-1}^{k_{-1}}d_0^{k_0}v\right): k_1,k_0, k_{-1},
 \cdots, k_{-n}\in \mathbb{Z}_+, n\in \mathbb{Z}_+\}.$$
Let us define a total order on $B$ as follows
$$\partial^{k_1}\otimes \left(d_{-n}^{k_{-n}}\cdots
d_{-1}^{k_{-1}}d_0^{k_0}v\right)<\partial^{l_1}\otimes
\left(d_{-m}^{l_{-m}}\cdots d_{-1}^{l_{-1}}d_0^{l_0}v\right) $$ if
and only if $(k_0,k_{-1},...k_{-n},0,0,..,0,
k_{1})<(l_0,l_{-1},...l_{-m},0,0,..,0, l_{1})$ in the
lexicographical order where the first zeros are $m$ copies and the
second zeros are $n$ copies, i.e.,
$$\aligned (a_1,a_2,...,a_{m+n+2})<&(b_1,b_2,...,b_{m+n+2}) \\ &\Longleftrightarrow (\exists
r>0)(a_i=b_i\forall i<r)(a_r<b_r).\endaligned$$

When we expand the elements in (4.5) into linear combinations in
terms of  $B$: $$d_{-n}^{k_{-n}}\cdots
d_{-1}^{k_{-1}}d_0^{k_0}d_1^{k_1}\cdot (1\otimes v)\hskip 5cm $$
$$\hskip 3cm =\lambda^{k_1}\partial^{k_1}\otimes \left(d_{-n}^{k_{-n}}\cdots
d_{-1}^{k_{-1}}d_0^{k_0}v\right)+{\text{lower terms}},$$ the leading
terms are exactly the corresponding basis elements in $B$. Thus
(4.5) is a basis for $W$. Therefore, $\tau$ is an isomorphism, i.e.,
$ \mathrm{Ind}_{\theta, \lambda}(\mathbb{C_\mathbf{m}}) \cong W.$
This is (i).

\vskip 5pt Solving (4.4) for $ \lambda_1, \lambda_2$ and $b$ we
obtain that
\begin{equation}\begin{split}
 \lambda_1&=\lambda^{-3}(2m_4-3\lambda m_3),\\
  \lambda _2&=\lambda^{-2}(m_4-2\lambda m_3+\lambda^2m_2),\\
  b&=1+\lambda^{-4}(m_4-\lambda m_3).\end{split}\end{equation}
  From
$\mathrm{Ind}_{\theta, \lambda}(\mathbb{C_\mathbf{m}}) \cong W$ we
see  that $\mathrm{Ind}_{\theta, \lambda}(\mathbb{C_\mathbf{m}})$ is
irreducible if and only if $W$ is irreducible; if and only if
$b\ne1$, and $\lambda_1\ne0$ or $\lambda_2\ne0$; if and only if
$m_4-\lambda m_3\neq 0$, and $2m_4-3\lambda m_3\neq0$ or
$m_4-2\lambda m_3+\lambda^2m_2\neq0.$ This is (ii) and completes the
proof.
\end{proof}

\vskip 5pt Note that Theorem 4 actually gives the  necessary and
sufficient conditions for the module $\mathrm{Ind}_{\theta,
\lambda}(\mathbb{C_\mathbf{m}}) $ described in \cite{MW} to be
irreducible. From Theorems 3 and 4, we have

\begin{proposition}
Let $\lambda\in \C^*$ and $b\in  \C\setminus\{1\}$ and let $V$  be
an infinite dimensional     module such that each $d_k$ is locally
finite on $V$ for all $k\ge n$ for a fixed $n\in\N$, and $V$ is not
a classical irreducible Whittaker module. Then $\Omega(\lambda,
b)\otimes V$  is a new non-weight irreducible $\V$ module.
\end{proposition}

\section{Applications}

\vskip 5pt In this section we will generalize the construction of
the Virasoro modules $\mathrm{Ind}_{\theta,
\lambda}(\mathbb{C_\mathbf{m}}) $ described in \cite{MW}.

\vskip 5pt Let  $n\in \Z_+$ and $ s_n, s_{n+1}, ..., s_{2n},
\lambda, \theta \in\C$ with $\lambda\ne0$. Let
$$\mathfrak{b}_{\lambda, n}:=\span_{\C}\{d_k-\lambda^{k-n+1}d_{n-1}:
k\ge n\},$$ which  is a subalgebra of
 $\V$.  Denote ${\bf {s}}=(s_n, s_{n+1}, ..., s_{2n})\in \C^{n+1}$. We define the action of
 $\mathfrak{b}_{\lambda, n}$
 on  $\C$  by the following
 \begin{equation}
 \begin{split}
&(d_k-\lambda^{k-n+1}d_{n-1})\cdot 1= s_k, \forall\,\,k=n, n+1, ...,2n;\\
&(d_k-\lambda^{k-n+1}d_{n-1})\cdot
1\\&\ \ =-(k-2n)s_{2n-1}\lambda^{k-2n+1}+(k-2n+1)s_{2n}\lambda^{k-2n},  \forall\,\,  k>2n,
\end{split}
 \end{equation}where we have assigned that $s_{-1}=0.$
 We denote the corresponding $\mathfrak{b}_{\lambda, n}$ module by $\mathcal {B}^{(n)}_{{\bf {s}}}$.  Note that
the second equation in (5.1) follows from the first.  Define the
 induced $\V$ module from $\mathcal {B}^{(n)} _{{\bf {s}}}$ as following
 \begin{equation} \mathrm{Ind}_{\theta, \lambda}(\mathcal {B}^{(n)} _{{\bf {s}}}):=
    \left(\mathrm{Ind}_{\mathfrak{b}_{\lambda,n}}^{\V}\mathcal{B}^{(n)}_{\mathbf{s}}\right )
    /(c-\theta)\left(\mathrm{Ind}_{\mathfrak{b}_{\lambda,n}}^{\V}\mathcal{B}^{(n)}_{\mathbf{s}}\right ).\end{equation}

We will determine necessary and sufficient conditions for
$\mathrm{Ind}_{\theta, \lambda}(\mathcal {B}^{(n)} _{{\bf {s}}})$ to
be simple  in the next three theorems for different cases of  $n$.
It is interesting to remark that the three cases are totally
different. Our main technique used here is Feigin-Fuchs' Thereom in
\cite{FF}, or Theorem A in \cite{A} (which is a refined version of
Feigin-Fuk's Theorem).

\vskip 5pt We first consider the case $n=0$. Obviously, the action (5.1) of
$\mathfrak{b}_{\lambda, 0}$ on $\mathcal {B}^{(0)}_{{\bf {s}}}$, where $\mathbf{s}=s_0\in \C,$ is
equivalent to the following action
\begin{equation}
(d_k-\lambda^k d_{0})(1\otimes v)= k\lambda^ks_0,\  k\geq -1.
\end{equation} For convenience, we shall take (5.3) in the case $n=0$. Recall that the Verma module $\bar{V}(\theta,0)$
over the Virasoro algebra was defined in the introduction. Now we
have the following

\begin{theorem}
Let $ {\bf s}=s_0,\theta\in \C,\lambda\in \C^* $ and denote $$M(\theta,
0)=\bar{V}(\theta,0)/U(\V)(d_{-1}(1+I(\theta,0))).$$
\begin{enumerate}[$(i).$]
\item The module $\mathrm{Ind}_{\theta, \lambda}(\mathcal {B}^{(0)}
_{{\bf{s}}})$ is isomorphic to
$\Omega(\lambda, b)\otimes M(\theta, 0)$, where $b=s_0+1$.
\item The module $\mathrm{Ind}_{\theta, \lambda}(\mathcal {B}^{(0)}
_{{\bf{s}}})$ is irreducible if and only if $s_0\neq 0$ and
$\theta\ne 1 - 6\frac{(p-q)^2}{pq}$ for any coprime integers $p,q
\ge2$.
\end{enumerate}
\end{theorem}

\begin{proof} Let $b\in \C.$
Denote by $$v=1+U(\V)(d_{-1}(1+I(\theta,0)))\in M(\theta,0).$$ Let
$\langle 1\otimes v\rangle $ be the submodule of $\Omega(\lambda,
b)\otimes M(\theta, 0)$ generated by $1\otimes v$. By repeatedly
acting $d_1$ on $1\otimes v$ we see that $\Omega(\lambda, b)\otimes
v\subseteq \langle 1\otimes v\rangle $. Then we  see that
$\Omega(\lambda, b)\otimes M(\theta, 0)$ is a cyclic module with
generator $1\otimes v$. From theorem 1 we can deduce that
$\Omega(\lambda, b)\otimes M(\theta, 0)$ is irreducible if and only
if $b\neq 1$ and $M(\theta,0)$ is irreducible.

Now let us consider the irreducibility of $M(\theta,0)$. From
Theorem A in \cite{A} we know that $M(\theta,0)$ is not irreducible
if and only if the maximal proper submodule $I(\theta, 0)$ of the
verma module $\bar{V}(\theta,0)$ cannot be generated by only one
singular vector; if and only if Conditions III$_-$ and III$_+$  in
\cite{A} are satisfied, if and only if
$\theta=\frac{(3p+2q)(3q+2p)}{pq}\in \C$, where the parameters $p,
q\in\C^*$ such that the straight line $ l_{\theta,0}: pk+ql-p-q=0$
in the plane $\C^2(k,l)$  contains
  infinite integral points $(k,l)$ with $kl>0$;   if and only if
$\theta= 1 - 6\frac{(p-q)^2}{pq}$ for any integers $p,q$ with $p,q
\geq2$ and $\gcd(p,q)=1$. Thus $\Omega(\lambda, b)\otimes M(\theta,
0)$ is irreducible if and only if  $s_0\neq 0$ and $\theta\ne 1 -
6\frac{(p-q)^2}{pq}$ for any integers $p,q$ with $p,q  \ge2$ and
$\gcd(p,q)=1$.

By simple computation we can obtain that
\begin{equation}
(d_k-\lambda^k d_{0})(1\otimes v)= k\lambda^k (b-1)(1\otimes v), k\geq -1.
\end{equation} If we set $s_0=b-1$, then there exists a  $\V$ module
homomorphism and hence epimorphism $\rho: \mathrm{Ind}_{\theta,
\lambda}(\mathcal {B}^{(0)} _{{\bf{s}}})\rightarrow \Omega(\lambda,
b)\otimes M(\theta,0)$ uniquely determined by $\rho(1)=1\otimes v$.
The same arguments used in the proof of Theorem 4 can deduce that
$\rho$ is a monomorphism and hence an isomorphism. Thus
$\mathrm{Ind}_{\theta, \lambda}(\mathcal {B}^{(0)} _{{\bf{s}}})\cong
\Omega(\lambda, b)\otimes M(\theta,0)$ and (i) holds.

\vskip 5pt Therefore, $\mathrm{Ind}_{\theta, \lambda}(\mathcal
{B}^{(0)} _{{\bf{s}}})$ is irreducible if and only if
$\Omega(\lambda, b)\otimes M(\theta,0)$ is irreducible. By $s_0=b-1$
and the irreducible conditions for $\Omega(\lambda, b)\otimes
M(\theta,0)$ we can deduce (ii). This competes the proof.
\end{proof}

We now handle the case  $ \mathrm{Ind}_{\theta,
\lambda}(\mathcal {B}^{(1)} _{{\bf{s}}})$ for ${\bf s}=(s_1,s_2)$.

\begin{theorem}
Let $\lambda \in \C^*$,   $\theta\in\C$ and ${\bf s}=(s_1,s_2)\in \C^2$.
  \begin{enumerate}[$(i).$]
\item The module $ \mathrm{Ind}_{\theta, \lambda}(\mathcal {B}^{(1)}
_{{\bf{s}}})$
  is isomorphic to $\Omega(\lambda, b)\otimes \bar V(\theta, h) ,$
  where
\begin{equation}
\begin{split}
& b=1+\lambda^{-2}(s_2-\lambda s_1),\\
& h=\lambda^{-2}(s_2-2\lambda s_1)
\end{split}
\end{equation}
\item The module $\mathrm{Ind}_{\theta,
\lambda}(\mathcal {B}^{(1)} _{{\bf{s}}})$  is
irreducible if and only if $s_2-\lambda s_1\neq 0$ and
\begin{equation}\begin{split}(\frac{s_2-2\lambda
s_1}{\lambda^{2}}+&\phi(k)+\frac{kl-1}{2})
 (\frac{s_2-2\lambda s_1 }{\lambda^{2}}+\phi(l)+\frac{kl-1}{2})\\&+\frac{(k^2-l^2)^2}{16}\neq 0,
  \ \forall k,l\in \N,\end{split}\end{equation} where
 $\phi(j)=\frac{(j^2-1)(\theta-13)}{24},\ j\in \N.$
\end{enumerate}
\end{theorem}

\begin{proof} Denote
$W=\Omega(\lambda,b)\otimes\bar{V}(\theta, h),$ where $ b, h\in\C$.
From Theorem 1 and the well-known Kac' determinant formula (see, for
example \cite{K1}) we can deduce that
\begin{enumerate}[$(a)$.]
\item $W$ is a cyclic module with generator $1\otimes v_0$, where $v_0$ is a highest weight vector
of $\bar{V}(\theta, h)$;
\item $W$ is irreducible if and only if $b\neq 1$ and for all $k,l \in \N$,
\begin{equation}(h+\phi(k)+\frac{(kl-1)}{2})(h+\phi(l)+\frac{(kl-1)}{2})+\frac{(k^2-l^2)^2}{16}
\neq 0,\end{equation}where
 $\phi(j)=\frac{(j^2-1)(\theta-13)}{24}, j\in \N.$
\end{enumerate}

\vskip 5pt By simple computation we can obtain the following equalities
\begin{equation}(d_k-\lambda^{k}d_0)(1\otimes v_0)=\lambda^k(k(b-1)-h)(1\otimes v_0), k\in \N .
 \end{equation}Set
 \begin{equation}\begin{split}
 &s_1=\lambda(b-1-h),\\
 &s_2=\lambda^2(2(b-1)-h).
 \end{split}
 \end{equation}  Solving the system (5.9) of equations for $b$ and $h$ we can get
 $b=1+\lambda^{-2}(s_2-\lambda s_1), h=\lambda^{-2}(s_2-2\lambda s_1).$
Moreover, we also have
 $$(d_k-\lambda^{k}d_0)(1\otimes v_0)=(-(k-2)s_1\lambda^{k-1}+(k-1)s_2\lambda^{k-2})(1\otimes v_0),\ k>2.$$
 So there exists a $\V$ module homomorphism and hence epimorphism
 $\sigma: \mathrm{Ind}_{\theta, \lambda}(\mathcal{B}^{(1)}_{{\bf{s}}})\rightarrow W$ uniquely determined by
 $\sigma(1)=1\otimes v_0$. It is not difficult to show (using the same method in the proof of
 Theorem 4) that $\sigma$ is actually
 injective and bijective. Thus, $\mathrm{Ind}_{\theta, \lambda}(\mathcal{B}^{(1)}_{{\bf{s}}})\cong W$.
 This is (i).

\vskip 5pt Therefore, $\mathrm{Ind}_{\theta,
\lambda}(\mathcal{B}^{(1)} _{{\bf{s}}})$ is irreducible if and only
if $W$ is irreducible; if and only if $b\neq 1$ and (5.7) holds; if
and only if
 $s_2-\lambda s_1\neq 0$ and

 \begin{equation}\begin{split}(\frac{s_2-2\lambda s_1}{\lambda^{2}}+&\phi(k)+\frac{kl-1}{2})
 (\frac{s_2-2\lambda s_1 }{\lambda^{2}}+\phi(l)+\frac{kl-1}{2})\\&+\frac{(k^2-l^2)^2}{16}\neq 0,
  \ \forall k,l\in \N,\end{split}\end{equation}where
 $\phi(j)=\frac{(j^2-1)(\theta-13)}{24},\ j\in \N.$  This is (ii) and completes the proof.
\end{proof}

Before treating the case  $ \mathrm{Ind}_{\theta, \lambda}(\mathcal
{B}^{(n)} _{{\bf {s}}}), n>1$, let us first recall the Whittaker
modules $L_{\psi_n, \theta}$  defined in \cite{LGZ}.

\vskip 5pt 
Let $n\in \N $ and  $(\lambda_{n}, \lambda_{n+1},\cdots,
\lambda_{2n})\in \C^{n+1}$. Define $\psi_{n}: \V
_+^{(n-1)}\rightarrow \C$ by the following \begin{equation}
\begin{split}&\psi_{n}(d_j)=\lambda_j, \ j=n,n+1, \cdots, 2n;\\
&\psi_n(d_j)=0, \ j>2n.\end{split}\end{equation} This can actually
define a $\V _+^{(n-1)}$ module action on $\C$ by $d_j\cdot
1=\psi_{n}(d_j),\ j\geq n$. Denote the $\V ^{(n-1)}_+$ module by
$\C_{\psi_n}.$ Then
\begin{equation} L_{\psi_n,\theta}=U(\V)\otimes_{U(\V _+^{(n-1)})}
\C_{\psi_n}/(c-\theta)U(\V)\otimes_{U(\V
_+^{(n-1)})}\C_{\psi_n}.\end{equation} From Theorem 7 in
\cite{LGZ} we know that $L_{\psi_n,\theta}$ is an irreducible $\V$
module if and only if $\lambda_{2n-1}\neq 0$ or $\lambda_{2n}\neq
0$. Moreover, it is easy to see that $L_{\psi_n,\theta}$ is a
locally nilpotent $\V _+^{(2n)}$ module.

\vskip 5pt For the modules $ \mathrm{Ind}_{\theta, \lambda}(\mathcal
{B}^{(n+1)} _{{\bf {s}}})$, where ${\bf s}=(s_{n+1}, s_{n+2},
\cdots, s_{2n+2})\in\C^{n+2}$ with $ n\ge 1$, we have the following

\begin{theorem}
Let $n\in\N$, ${\bf s}=(s_{n+1}, s_{n+2}, \cdots,
s_{2n+2})\in\C^{n+2}, \theta \in \C$ and $ \lambda\in\C^*$.
  \begin{enumerate}[$(i).$]
\item The module $ \mathrm{Ind}_{\theta, \lambda}(\mathcal {B}^{(n+1)} _{\bf {s}})$
  is isomorphic to $\Omega(\lambda, b)\otimes L_{\psi_n, \theta} ,$
  where
\begin{equation}
\begin{split}
& b=1+\lambda^{-2n-2}(s_{2n+2}-\lambda s_{2n+1}),\\
& \lambda_n=\lambda^{-n-2}((n+1)s_{2n+2}-(n+2)\lambda s_{2n+1}),\\
&\lambda_k=s_k-\lambda^{k-2n-2}(-(k-2n-2)\lambda
s_{2n+1}+(k-2n-1)s_{2n+2}), \\&\ \ \ \ n+1\leq k\leq 2n.
\end{split}
\end{equation}
\item The module $\mathrm{Ind}_{\theta,
\lambda}(\mathcal {B}^{(n+1)} _{\bf {s}})$ is irreducible if and only if
\begin{equation}
s_{2n+2}-\lambda s_{2n+1}\neq 0,{\text{ and }}\end{equation}
\begin{equation}
s_{2n-1}\neq \lambda^{-3}(3\lambda s_{2n+1}-2s_{2n+2})\ {\text{ or
}} \, s_{2n}\neq \lambda^{-2} (2\lambda s_{2n+1}-s_{2n+2}),
\end{equation}
where we have assumed that  $s_1=0$ if $n=1$.
\end{enumerate}

\end{theorem}

\begin{proof} Let $(\lambda_{n}, \lambda_{n+1},\cdots, \lambda_{2n})\in \C^{n+1}, \lambda_k=0, k>2n,$
and $L_{\psi_n,\theta}$ be defined by (5.11) and (5.12). Denote
$$v=1+(c-\theta)U(\V)\otimes_{U(\V _+^{(n-1)})}\C_{\psi_n}\in
L_{\psi_n,\theta}$$ and $W=\Omega(\lambda, b)\otimes L_{\psi_n,
\theta}$, where $b\in \C$. Then by Theorem 1 and the irreducible
conditions for the Whittaker module $L_{\psi_n,\theta}$ we can
deduce the following
\begin{enumerate}[$(a)$.]
\item $W$ is a cyclic module with generator $1\otimes v$;
\item $W$ is irreducible if and only if $b\neq 1$ and $\lambda_{2n-1}\neq 0$ or $\lambda_{2n}\neq 0.$
\end{enumerate}

\vskip 5pt For $k>n$ we compute
$$
\aligned
&\ \ \ \ (d_k-\lambda^{k-n}d_n)(1\otimes v) \\
&=(d_k-\lambda^{k-n}d_n)(1)\otimes v + 1\otimes (d_k-\lambda^{k-n}d_n)(v)\\
& =(\lambda^k(\partial+k(b-1))-\lambda^k(\partial+n(b-1)))\otimes v+1\otimes (\lambda_k-\lambda^{k-n}\lambda_n)v\\
&=(\lambda^{k}(k-n)(b-1)+(\lambda_k-\lambda^{k-n}\lambda_n))(1\otimes
v).
\endaligned
$$ Set
\begin{equation}s_k=\lambda^{k}(k-n)(b-1)+(\lambda_k-\lambda^{k-n}\lambda_n), \ n+1\leq k\leq 2n+2.\end{equation}
Then solving the system (5.16) of equations for $b,\lambda_{n},
\lambda_{n+1},\cdots, \lambda_{2n}$ we can get
\begin{equation}
\begin{split}
& b=1+\lambda^{-2n-2}(s_{2n+2}-\lambda s_{2n+1}),\\
& \lambda_n=\lambda^{-n-2}((n+1)s_{2n+2}-(n+2)\lambda s_{2n+1}),\\
&\lambda_k=s_k-\lambda^{k-2n-2}(-(k-2n-2)\lambda s_{2n+1}+(k-2n-1)s_{2n+2}), \\&\ \ \ \ n+1\leq k\leq 2n.
\end{split}
\end{equation}
By simple computation we can obtain
$$\aligned
& (d_k-\lambda^{k-n}d_{n})(1\otimes v)\\
& =(-(k-2n-2)s_{2n+1}\lambda^{k-2n-1}+(k-2n-1)s_{2n+2}\lambda^{k-2n-2})(1\otimes v),
\endaligned$$
where $k>2n+2$. Using similar arguments  in the proof of Theorem 4,
we deduce  that $\mathrm{Ind}_{\theta, \lambda}(\mathcal {B}^{(n+1)}
_{\bf {s}})\cong \Omega(\lambda, b) \otimes L_{\psi_n,\theta}$. This
is (i).

\vskip 5pt Therefore, $\mathrm{Ind}_{\theta, \lambda}(\mathcal
{B}^{(n+1)} _{\bf {s}})$
 is irreducible if and only if
$\Omega(\lambda, b)\otimes L_{\psi_n,\theta}$ is irreducible; if and only
if $b\neq 1$, and $\lambda_{2n-1}\neq0$ or $\lambda_{2n}\neq 0$; and if and only if (5.14)
 and (5.15) hold. This implies (ii) and completes the proof.
\end{proof}

\vskip 5pt Note that the modules $\mathrm{Ind}_{\theta,
\lambda}(\C_{\mathbf{m}})$ defined in \cite{MW} are just the
modules $ \mathrm{Ind}_{\theta, \lambda}(\mathcal {B}^{(2)} _{{\bf{s}}})$
we defined here for ${\bf s}=(m_2,m_3,m_4)$.\\

Now we will characterize the submodules of $\mathrm{Ind}_{\theta,
\lambda}(\mathcal {B}^{(n)} _{\bf {s}})$. Because of Theorems 7, 8,
and 9, it is enough to consider   submodules of  $\Omega(\lambda, b
)\otimes V$
 where  $V $ is determined by Theorems 7,8, 9, which is the
following

\begin{theorem}
Let $n\in \N$, $\lambda\in \C^*,b, \theta\in\C$, and let $V$ be a
highest weight module or  $ L_{\psi_n,\theta}$ over  $\V$.
\begin{enumerate}[$(i)$.]
\item If $b\neq 1,$ then each submodule $M $ of $\Omega(\lambda, b
)\otimes V$ is of the form $\Omega(\lambda, b )\otimes X$ for some
submodule
 $X$ of $V $.

\item If $b=1,$ then each submodule $M $ of $\Omega(\lambda, b
)\otimes V$ is of the form $\partial\Omega(\lambda,1)\otimes
X_1+\Omega (\lambda,1)\otimes X_2 $  where $X_1$ and $X_2$ are
submodules of $V $.
\end{enumerate}
\end{theorem}

\begin{proof}
(i).  $b \neq 1$.

Then $\Omega(\lambda, b )$ is  an irreducible $\V$ module. Let $Y$
be a nonzero submodule of $\Omega(\lambda,b )\otimes V .$

{\bf Claim 1.} {\it If $\sum^s_{j=0}\partial^j\otimes v_j\in Y,$
 where $v_j\in V $, then $\Omega(\lambda,b )\otimes v_j\subset Y$ for all $ j=0,1,\cdots,
 s.$}

 Using the same arguments in the proof of Theorem 1 we can deduce that
$\Omega(\lambda, b )\otimes v_s\subset Y$ and hence
$\Omega(\lambda,b )\otimes v_j\subset Y, j=0,1,\cdots, s,$ by
induction on $j$.

{\bf Claim 2.} $Y=\Omega(\lambda,b )\otimes X$, where $X$ is a
submodule of $V $.

Let $X$ be the maximal subspace of $V $ satisfying $\Omega(\lambda,b
)\otimes X\subset Y$. The maximality of $X$ forces that $X$ is a
submodule of $V $. Using Claim 1 we see that $\Omega(\lambda,b
)\otimes X=Y.$

Thus (i) follows.

\

(ii). Now consider the case  $b =1$.

Let $Z$ be a   submodule of $\Omega(\lambda,b )\otimes V .$ Take a
nonzero $w=\sum_{j=0}^{s}\partial^j\otimes
    v_j\in Z$ such that
      $v_j\in V$.

\vskip 5pt {\bf Claim 3.} $\Omega(\lambda,1)\otimes u_0\subseteq Z$
and $\partial\Omega(\lambda,1)\otimes u_i\subset Z$ for all $i\ge
1$.

\vskip 5pt We will prove this by induction on $s$. This is true for
$s=0$ by simple computations. Now suppose  $s>0$. Let
$K=\max\{K(v_j):j=0,1,\cdots, s\}$. Using $d_l(v_j)=0$ for all
$l\geq K$
      and $j=0,1,\cdots,s$, we deduce that
$$d_lw=\sum_{j=0}^{s}\lambda^l\partial(\partial-l)^j\otimes
    v_j\in Z, \,\,\forall l\ge K.$$
Then the coefficient of $l^{s}$ is $\lambda^l\partial\otimes v_s$
which has to be in $Z$. By simple computations we deduce that
$\partial\Omega(\lambda,1)\otimes u_s\subset Z$. Claim 3 follows.

\vskip 5pt

Let $X_1, X_2$ be maximal subspaces of $Z$ such that $
\Omega(\lambda,1)\otimes X_1+\partial\Omega(\lambda,1)\otimes
X_2\subseteq Z$. It is easy to see that $X_1$ and $X_2$ are
submodules of $V$. Using Claim 3 it is not hard to deduce that
$Z=\partial\Omega(\lambda,1)\otimes X_1+\Omega(\lambda,1)\otimes
X_2.$ This is (ii) and completes the proof.
\end{proof}

Note that when $\theta=0$, the modules $\mathrm{Ind}_{\theta,
\lambda}(\mathcal {B}^{(0)} _{{\bf{s}}})$ are exactly the highest
weight-like modules defined in \cite{GLZ}. Now we can answer the
open problem in \cite{GLZ}: if ${\bf{s}}=s_0\in \C^*,$ whether or not $\mathrm{Ind}_{0,
\lambda}(\mathcal {B}^{(0)} _{{\bf{s}}})$ has a unique maximal
submodule. From (i) of Theorem 9 we know that the maximal submodules of
$\mathrm{Ind}_{\theta,
\lambda}(\mathcal {B}^{(0)} _{{\bf{s}}})$ correspond to the maximal submodules of $M(\theta,0)$.
Since $M(\theta,0)$ is a highest weight module, it has a unique maximal submodule, so does $\mathrm{Ind}_{\theta,
\lambda}(\mathcal {B}^{(0)} _{{\bf{s}}}).$ In the special case $\theta=0,$ the conclusion
is certainly true. Therefore,
we have the following
\begin{corollary} Let $\lambda, {\bf s}=s_0\in \C^*.$ Then
$\mathrm{Ind}_{0, \lambda}(\mathcal {B}^{(0)} _{{\bf{s}}})$ has a unique maximal submodule.
\end{corollary}

\noindent {\bf Acknowledgement.} The second author is partially
supported by NSF of China (Grant 10871192) and NSERC. The authors
want to thank Prof. R. Lu for a lot of helpful discussions when they
were preparing the paper, also to thank Prof. V. Mazorchuk for
scrutinizing the old version and giving a lot of nice suggestions.

\vspace{1mm}

\noindent H.T.: College of Mathematics and Information Science,
Hebei Normal (Teachers) University, Shijiazhuang, Hebei, 050016 P.
R. China, and Department of Applied Mathematics, Changchun
University of Science and Technology, Changchun, Jilin, 130022, P.R.
China. Email: tanhj999@yahoo.com.cn

\vspace{0.2cm} \noindent K.Z.: Department of Mathematics, Wilfrid
Laurier University, Waterloo, ON, Canada N2L 3C5, and College of
Mathematics and Information Science, Hebei Normal (Teachers)
University, Shijiazhuang, Hebei, 050016 P. R. China. Email:
kzhao@wlu.ca

\end{document}